\documentclass[10pt]{amsart} 
\usepackage{tikz-cd}
 \newcommand{\filename}{{dynamical-degrees-Vojta-24-June-2024}} 
\usepackage{amsmath,amsthm,amssymb,mathrsfs}
 
\renewcommand{\geq}{\geqslant}
\renewcommand{\leq}{\leqslant}
\newcommand{\Osh}{{\mathcal O}}                        %  Structure sheaf
\newcommand{\Vol}{\operatorname{Vol}}

\newcommand{\KK}{\mathbf{K}}

\newcommand{\PP}{\mathbb{P}} % projective space
\newcommand{\RR}{\mathbb{R}} % real numbers
\newtheorem{theorem}{Theorem}[section]

\newtheorem{corollary}[theorem]{Corollary}

\theoremstyle{definition}

\newtheorem{remark}[theorem]{Remark}
\newtheorem{example}[theorem]{Example}

\newtheorem{problem}[theorem]{Problem}
\numberwithin{equation}{section}

\begin{document}

\author{Nathan Grieve}
\address{Department of Mathematics \& Statistics,
Acadia University, Huggins Science Hall, Room 130,
12 University Avenue,
Wolfville, Nova Scotia, B4P 2R6
Canada; 
School of Mathematics and Statistics, 4302 Herzberg Laboratories, Carleton University, 1125 Colonel By Drive, Ottawa, ON, K1S 5B6, Canada; 
D\'{e}partement de math\'{e}matiques, Universit\'{e} du Qu\'{e}bec \`a Montr\'{e}al, Local PK-5151, 201 Avenue du Pr\'{e}sident-Kennedy, Montr\'{e}al, QC, H2X 3Y7, Canada;
Department of Pure Mathematics, University of Waterloo, 200 University Avenue West, Waterloo, ON, N2L 3G1, Canada
}

\email{nathan.m.grieve@gmail.com}%

\author{Chatchai Noytaptim}

\address{
Department of Pure Mathematics, University of Waterloo, 200 University Avenue West, Waterloo, ON, N2L 3G1, Canada
}

\email{cnoytaptim@uwaterloo.ca}

\thanks{
The first author thanks the Natural Sciences and Engineering Research Council of Canada for their support through his grants DGECR-2021-00218 and RGPIN-2021-03821. \\
\\
Date: \today.  \\
File name: \filename
}

\title[On non-Zariski density of $(D,S)$-integral points in forward orbits]{On non-Zariski density of $(D,S)$-integral points in forward orbits and the Subspace Theorem}

\begin{abstract}
Working over a base number field $\KK$, we study the attractive question of Zariski non-density for $(D,S)$-integral points in $\mathrm{O}_f(x)$ the forward $f$-orbit of a rational point $x \in X(\KK)$.  Here, $f \colon X \rightarrow X$ is a regular surjective self-map for $X$ a geometrically irreducible projective variety over $\KK$.  Given a non-zero and effective $f$-quasi-polarizable Cartier divisor $D$ on $X$ and defined over $\KK$, our main result gives a sufficient condition, that is formulated in terms of the $f$-dynamics of $D$, for non-Zariski density of certain dynamically defined subsets of $\mathrm{O}_f(x)$.  For the case of $(D,S)$-integral points, this result gives a sufficient condition for non-Zariski density of integral points in $\mathrm{O}_f(x)$.  Our approach expands on that of Yasufuku, \cite{Yasufuku:2015}, building on earlier work of Silverman \cite{Silverman:1993}.  Our main result gives an unconditional form of the main results of loc.~cit.; the key arithmetic input to our main theorem is the Subspace Theorem of Schmidt in the generalized form that has been given by Ru and Vojta in \cite{Ru:Vojta:2016} and expanded upon in \cite{Grieve:points:bounded:degree} and \cite{Grieve:qualitative:subspace}.
\end{abstract}

\maketitle

\section{Introduction}

Our purpose here, is to study the attractive problem of Zariski non-density for $(D,S)$-integral points in $\mathrm{O}_f(x)$ the forward $f$-orbit of a rational point $x \in X(\KK)$.  We work over a base number field $\KK$, with fixed choice of algebraic closure $\overline{\KK}$, and $f \colon X \rightarrow X$ is a regular surjective self-map for $X$ a geometrically irreducible projective variety over $\KK$.  

As in \cite{Meng:Zhang:2018},  we say that $f$ is \emph{polarizable} over $\KK$ if there is an ample Cartier divisor $H$ on $X$ and defined over $\KK$ which is such that, over the base change of $X$ to $\overline{\KK}$, the pullback of $H$ with respect to $f$ is linearly equivalent to some positive integral multiple of $H$.  

More generally, we say that $f$ is \emph{quasi-polarizable} over $\KK$ if there is a nef and big Cartier divisor $M$ on $X$ and defined over $\KK$ which is such that, working over $\overline{\KK}$, the pullback of $M$ with respect to $f$ is linearly equivalent to some positive integral multiple of $M$.

If $f$ is quasi-polarizable, then, by \cite[Proposition 3.6]{Meng:Zhang:2018}, $f$ is polarizable and, in particular, there is an integer $\delta_f > 1$, the \emph{dynamical degree} of $f$, which is such that 
$$f^* L \equiv_{\mathrm{lin.}} \delta_f L$$ 
for all $f$-quasi-polarizable line bundles $L$ on $X$ and defined over $\KK$.

To place matters into perspective, it is helpful to recall that the condition that $X$ admit an $f$-quasi-polarization prevents the Kodaria dimension from being positive \cite[Proposition 2.2.1]{Zhang:2006}.  (See \cite[Proposition 2.3.1]{Zhang:2006}, for example, for a more precise classification for the case that $X$ has dimension $2$.)

For a point $x \in X$, its \emph{$f$-orbit} is defined to be 
$$
\mathrm{O}_f(x) := \{ f^{(n)}(x) : n \geq 0 \} \text{.}
$$
Here $f^{(n)} \colon X \rightarrow X$ is the $n$-fold iterate of $f$.

We want to study the following attractive problem.

\begin{problem}\label{non-Zariski-density:subsets-f-orbit}
Let $f \colon X \rightarrow X$ be a regular surjective self-map for $X$ a geometrically irreducible projective variety over a number field $\KK$.
Given $x \in X$, find criteria for non-Zariski density for $(D,S)$-integral points in the \emph{$f$-orbit} $\mathrm{O}_f(x)$.
\end{problem}
   
In Problem \ref{non-Zariski-density:subsets-f-orbit}, $D$ is a non-zero effective Cartier divisor on $X$ and $S$ is a finite set of places of $\KK$ and not necessarily containing all infinite places.   The \emph{$(D,S)$-integral points} of $X$ are those points $x' \in X$ which have the property that
$$
m_S(D,x') = h_{\Osh_X(D)}(x') + \mathrm{O}(1) \text{.}
$$

Here, $m_S(D,\cdot)$ is a proximity function for $D$ with respect to $S$ and a choice of presentation for $D$ and $h_{\Osh_X(D)}(\cdot)$ is the logarithmic height function.  We refer to Section \ref{prelims}, for more details in regards to our conventions about proximity functions $m_S(D,\cdot)$, counting functions $n_S(D,\cdot)$ and logarithmic height functions.

On main result here, Theorem \ref{orbit:non:density} below, is in the direction of Problem \ref{non-Zariski-density:subsets-f-orbit}.  In particular, it gives an unconditional and more general formulation of the main results of \cite{Yasufuku:2015}.   In the formulation of the conclusion in Theorem \ref{orbit:non:density}, we refer to Section  \ref{prelims} for precise details about what we mean by \emph{reduced properly intersecting part} of the $n$th iterate of a quasi-polarizable big and nef divisor.

\begin{theorem}\label{orbit:non:density}
Working over a base number field $\KK$, let $X$ be a geometrically irreducible projective variety.  Let $f \colon X \rightarrow X$ be a regular surjective self-map with $\delta_f > 1$.  Let $L$ be a big and nef $f$-quasi-polarizable divisor on $X$.  Let $D$ be a non-zero and effective $f$-quasi-polarizable Cartier divisor on $X$ and defined over $\KK$.  Assume that $D$ is linearly equivalent to $mL$ for some $m>0$.  For $n \geq 1$, let 
$$D^{(n)} = (f^{(n)})^*D$$ 
and let $(D^{\operatorname{red}}_{\operatorname{p.i.}})^{(n)}$ be the reduced properly intersecting part of $D^{(n)}$.  Write 
$$(D^{\operatorname{red}}_{\operatorname{p.i.}})^{(n)} = D_1^{(n)}+\dots + D_{q_n}^{(n)}$$ for $D_i^{(n)}$ distinct, irreducible, reduced and properly intersecting Cartier divisors on $X$.  Assume that $D_i^{(n)}$ is linearly equivalent to $m_i^{(n)}L$ for positive integers $m_i^{(n)}$.  

Set 
$$\gamma = \gamma((D^{\operatorname{red}}_{\operatorname{p.i.}})^{(n)}) := \left( \max_i \left\{ m_i^{(n)} \right\}\right) \left( \dim X + 1\right)$$
and
$$
c_n := \frac{\sum_{i=1}^{q_n} m_i^{(n)} - \gamma}{\delta^n_f m^n} \text{.}
$$

Let $x \in X(\KK)$.  
Then for each finite set $S \subset M_{\KK}$ and all sufficiently small  $\epsilon > 0$, the set 
$$
\left\{ f^{(m)}(x) : \frac{n_S(D,f^{(m)}(x)) }{m h_L(f^{(m)}(x)) } \leq c_n - \epsilon \right\} \subseteq \mathrm{O}_f(x)
$$
is Zariski-non-dense.  
\end{theorem}

Examples \ref{Yasufuku:2015:Eg} and \ref{Eg:2} below illustrate the novelty of Theorem \ref{orbit:non:density} and, in particular, highlight the fact that Theorem \ref{orbit:non:density} serves to give an unconditional form of the main results from \cite{Yasufuku:2015} (compare also with the main results from \cite{Silverman:1993}).  These examples also serve to illustrate, in a concrete way, that it is possible to achieve
$$
c_n > 0 \text{;}
$$
this is a key point to ensuring that the conclusion of Theorem \ref{orbit:non:density} is non-empty in practice.

\begin{example}[Compare with {\cite[Example 1]{Yasufuku:2015}}]\label{Yasufuku:2015:Eg}
Let 
$$f \colon \PP^2_{[x:y:z]} \rightarrow \PP^2_{[x:y:z]}$$ 
be defined by the condition that 
$$[x:y:z] \mapsto [y^4 + z^4 : x^3(x+y+z):yz^3].$$  Let $D$ be the divisor that is defined by the condition that $z = 0$.  Then 
$$
D^{(1)} = (f^{(1)})^* D
$$
is the divisor 
$$
yz^3 = 0
$$
and thus 
$$\delta_f = 4 \text{.}$$
On the other hand
$$D^{(2)} = (f^{(2)})^* D$$ is the divisor that is defined by the condition that 
$$x^3(x+y+z)y^3z^9 = 0.$$  
whereas $(D^{\operatorname{red}}_{\operatorname{p.i.}})^{(2)}$ is the divisor that is defined by the equation 
$$xyz(x+y+z) = 0.$$  
Thus, the conclusion of Theorem \ref{orbit:non:density} applied to $f$, $D$, $L=\Osh_{\PP^2}(1)$ and $n = 2$ with $c_2 = \frac{1}{16}$ is satisfied.
\end{example}

\begin{example}\label{Eg:2} Fix a collection of four general (and hence irreducible) degree $2$ polynomials 
$$p,q,r,s \in \KK[x,y,z]$$ 
and let 
$$f \colon \PP^2_{[x:y:z]} \rightarrow \PP^2_{[x:y:z]}$$ 
be defined by the condition that 
$$[x:y:z] \mapsto [pq:rs:x^2y^2] \text{.}$$
Let $D$ be the divisor that is defined by the condition that $z = 0$.  Then 
$$
D^{(1)} = (f^{(1)})^* D
$$
is the divisor 
$$
x^2y^2 = 0
$$
and thus 
$$\delta_f = 4 \text{.}$$
On the other hand
$$D^{(2)} = (f^{(2)})^* D$$ is the divisor that is defined by the condition that 
$$p^2q^2r^2s^2 = 0.$$  
whereas $(D^{\operatorname{red}}_{\operatorname{p.i.}})^{(2)}$ is the divisor that is defined by the equation 
$$pqrs = 0.$$  
Thus, the conclusion of Theorem \ref{orbit:non:density} applied to $f$, $D$, $L=\Osh_{\PP^2}(1)$ and $n = 2$ with $c_2 = \frac{1}{8}$ is satisfied.
\end{example}

As a further illustration of Theorem \ref{orbit:non:density} we make note that it gives a sufficient condition for non-Zariski density of $(D,S)$-integral points in forward orbits.  We make this statement precise in Corollary \ref{orbit:non:density:cor} below.

\begin{corollary}\label{orbit:non:density:cor}
In the setting of Theorem \ref{orbit:non:density}, the collection of $(D,S)$-integral points in $\mathrm{O}_f(x)$ are non-Zariski dense if $c_n>0$ for some $n>0$.
\end{corollary}

\begin{proof}
Recall, that the $(D,S)$-integral points $x' \in X$ are characterized by the condition that 
$$m_S(D,x') = h_{\Osh_X(D)}(x') + \mathrm{O}(1) \text{.}$$  
So, if $c_n>0$ for some $n>0$, then Theorem \ref{orbit:non:density} implies that all $(D,S)$-integral points in $\mathrm{O}_f(x)$ with sufficiently large height are non-Zariski dense.  The conclusion of Corollary \ref{orbit:non:density:cor} then follows in light of the Northcott property for big line bundles.
\end{proof}

The proof of Theorem \ref{orbit:non:density} employs closely the techniques from \cite{Yasufuku:2015}.  The key point to obtaining unconditional forms of the results proven there is an application of Theorem \ref{EF:RV:Subspace}, which is a result of some independent interest, in place of Vojta's Main Conjecture.  In turn, Theorem \ref{EF:RV:Subspace} is deduced from the Arithmetic General Theorem of Ru and Vojta (see for example \cite[Arithmetic General Theorem, p.~961]{Ru:Vojta:2016}, \cite[Theorem 1.1]{Grieve:points:bounded:degree} or \cite[Theorem 1.6]{Grieve:qualitative:subspace} for complementary formulations of that result).

\subsection*{Acknowledgements}  The first author thanks NSERC for their support through his grants DGECR-2021-00218 and RGPIN-2021-03821.  Both authors thank colleagues for their interest and discussions on related topics.  The authors thank an anonymous referee for carefully reading our article and for providing helpful comments.

\section{Notations and conventions}\label{prelims}

Our conventions about absolute values, local Weil and height functions are identical to those of \cite{Bombieri:Gubler}.  We note that a more general version of that theory, so as to allow for a theory of local Weil functions for divisors which have field of definition with respect to some finite extension of the base number field, and normalized with respect to the base number field in question,  is developed in \cite{Grieve:qualitative:subspace}.  

On the other hand, that level of generality, which is important in some applications, turns out to be not entirely well suited for our purposes here.  The reason is that we want to relate the proximity and counting function to the height function and, for that purpose, it is helpful to consider only the case of those Cartier divisors which have field of definition of the given base number field.

To describe our notations and conventions here in more precise terms, if $\KK$ is a number field, then $M_{\KK}$ denotes its set of places and $|\cdot|_v$ for $v \in M_{\KK}$ is its corresponding absolute value as described in \cite[p.~11]{Bombieri:Gubler}.  In particular the product formula is satisfied with multiplicities equal to one.

If $X$ is a geometrically irreducible projective variety over $\KK$ and if $D$ is a Cartier divisor on $X$, then $\lambda_v(D,\cdot)$ for $v \in M_{\KK}$ denotes a local Weil function for $D$ that is determined by fixing a choice of presentation for $D$.  Further, if $L = \Osh_X(D)$, then $h_L(\cdot)$ denotes the logarithmic height function \cite[Section 2.3]{Bombieri:Gubler}.   

Fixing a finite set of places $S \subset M_{\KK}$ 
$$n_S(D,\cdot) := \sum_{v \in M_{\KK} \setminus S} \lambda_v(D,\cdot)$$ 
denotes the counting function whereas 
$$m_S(D,\cdot) := \sum_{v \in S} \lambda_v(D,\cdot)$$ 
denotes the proximity function.  By this notation 
$$h_{\Osh_X(D)}(\cdot) = n_S(D,\cdot) + m_S(D,\cdot) + \mathrm{O}(1) \text{.}$$

Let $D_1,\dots,D_q$ be a collection of nonzero effective Cartier divisors on $X$ and defined over $\KK$.  As in \cite[Definition 2.1]{Ru:Vojta:2016}, we say that $D_1,\dots,D_q$ \emph{intersect properly} if for all $$I \subseteq \{1,\dots,q\}$$ and all $$x \in \bigcap_{i \in I} \operatorname{Supp}(D_i)$$ the local defining equations in the local ring $\Osh_{X,x}$ of the divisors $D_i$ form a regular sequence.

If $D$ is a non-zero effective Cartier divisor on $X$, then we let $D^{\mathrm{red}}_{\operatorname{p.i.}}$ denote its \emph{reduced properly intersecting part}.  By this we mean that writing $D$ as a sum of distinct irreducible and reduced divisors with multiplicities
$$D = \sum_{i=1}^q m_i D_i $$ 
then
$$D^{\mathrm{red}}_{\operatorname{p.i.}} \leq D$$ is the divisor 
$$D^{\mathrm{red}}_{\operatorname{p.i.}} = \sum_{j \in J} D_j$$
where 
$J \subseteq \{1,\dots,q\}$ 
is the largest subset such that the collection of divisors $D_j$ with $j \in J$ intersect properly.

This concept of reduced properly intersecting part of a Cartier divisor, which is helpful for our purposes here, can be compared with the concept of \emph{normal crossings part} from \cite{Yasufuku:2015}.

\section{Diophantine arithmetic preliminaries}\label{diohpantine:prelims}

The key Diophantine arithmetic inequality that gets used in the proof of Theorem \ref{orbit:non:density} is Theorem \ref{EF:RV:Subspace} below.  It is a generalized form of \cite[Theorem C]{Ru:Vojta:2016} (compare also with \cite[Theorems 1.1 and 1.3]{Heier:Levin:2017}).  Theorem \ref{EF:RV:Subspace} is a consequence of \cite[Arithmetic General Theorem, p.~961]{Ru:Vojta:2016} (see also, for example, \cite[Theorem 1.1]{Grieve:points:bounded:degree} and \cite[Theorem 1.6]{Grieve:qualitative:subspace} for alternative and more general forms of that result).

\begin{theorem}\label{EF:RV:Subspace}
Let $X$ be a projective variety over a number field $\KK$.  Let $D_1,\dots,D_q$ be a collection of properly intersecting effective divisors on $X$.  Let $S \subset M_{\KK}$ be a finite set of places.  Assume that there exists a big and nef line bundle $L$ on $X$ and positive integers $d_i$ such that the class of $D_i$ is linearly equivalent to the class of $d_iL$ for $i = 1,\dots,q$. 
Then for all $\epsilon > 0$, the inequality
$$
\sum_{i=1}^q \frac{1}{d_i} m_S(x,D_i) \leq \left( \dim X + 1 + \epsilon \right) h_L(x)
$$
is valid for all $\KK$-rational points outside of some proper Zariski closed subset of $X$.
\end{theorem}
\begin{proof}
Using the fact that $L$ is big and nef and 
$$D_i \equiv_{\mathrm{lin.}} d_i L\text{ for $i=1,\dots q$}$$ 
we note that 
\begin{equation}\label{asymp:vanishing:constant}
\beta(L,D_i) = \beta(D_i) = \frac{1}{d_i (\dim X +1)}
\end{equation}
for
$$
\beta(L,D_i) := \liminf_{m\to \infty} \frac{\sum_{\ell \geq 1} h^0(X,mL-\ell D_i)}{m h^0(X,mL)} = \int_0^\infty \frac{\operatorname{Vol}(L-tD_i)}{\Vol(L)} \mathrm{d}t \text{.}
$$
Here $\Vol(\cdot)$ denotes the volume of an $\RR$-Cartier divisor class \cite[Definition 2.2.31]{Laz}.

The desired conclusion then follows from \cite[Arithmetic General Theorem, p. 961]{Ru:Vojta:2016}. (See also \cite[Theorem 1.1]{Grieve:points:bounded:degree} and \cite[Theorem 1.6]{Grieve:qualitative:subspace} for alternative and more general forms of that result and deduced from the Subspace Theorem in the form of \cite[Theorem 7.2.2]{Bombieri:Gubler} and \cite[Parametric Subspace Theorem, p.~515]{Evertse:Ferretti:2013}, respectively, and using our conventions about absolute values here.)
\end{proof}

\begin{remark}
The quantity \eqref{asymp:vanishing:constant} gives the expected order of vanishing of $D_i$ with respect to $L$.  Recall, that it admits a calculation via the theory of Newton-Okounkov bodies, see \cite[Theorem 1.1]{Grieve:MVT:2019} and \cite[Theorem 6.2]{Grieve:chow:approx}.  Note that these results from \cite{Grieve:MVT:2019} and \cite{Grieve:chow:approx} apply quite generally.  For instance, they apply to the case that $L$ is big as opposed to being big and nef.  Further,  each of the $D_i$ are not required to be linearly equivalent to some positive multiple of $L$.  Also for the case of inflectionary embeddings of very ample linear systems the quantity \eqref{asymp:vanishing:constant} admits a description as a normalized Chow weight \cite[Theorem 1.1]{Grieve:chow:approx}.  
\end{remark}

\begin{remark}
A more general form of Theorem \ref{EF:RV:Subspace} holds true; it allows the coefficients of the divisors $D_i$ to be defined over a finite extension of the base number field $\KK$.  Here, for our present purposes, that level of generality does not give us any additional benefit.  Indeed, the reason, is that the conclusion of Theorem \ref{orbit:non:density} is formulated in terms of the counting function $n_S(D,\cdot)$; in order to relate it to the height function $h_{\Osh_X(D)}(\cdot)$, following our normalization conventions for absolute values which are consistent, with those of \cite{Bombieri:Gubler}, \cite{Grieve:points:bounded:degree} and \cite{Grieve:qualitative:subspace}, it is helpful to have $D$ defined over $\KK$.
\end{remark}

\section{Proof of Theorem \ref{orbit:non:density}}

Our approach to proving Theorem \ref{orbit:non:density} builds on the strategy from \cite{Yasufuku:2015} and \cite{Silverman:2013}.  A key point is to replace the use of Vojta's Main Conjecture by Theorem \ref{EF:RV:Subspace}.   In particular, Theorem \ref{orbit:non:density} holds true unconditionally and without any appeal to Vojta's Main Conjecture. 

\begin{proof}[Proof of Theorem \ref{orbit:non:density}]
We apply Theorem \ref{EF:RV:Subspace}. 
The conclusion, perhaps after replacing $\epsilon > 0$ by some small $\epsilon' > 0$, if required, is that the inequality
\begin{equation}\label{eqn:1}
\sum_{v \in S} \lambda_v((D^{\operatorname{red}}_{\operatorname{p.i.}})^{(n)} , x' ) \leq ( \gamma + \epsilon ) h_L(x')
\end{equation}
is valid for all $x' \in X \setminus Z$ for $Z \subsetneq X$ some proper Zariski closed subset.

By assumption
$$D \equiv_{\mathrm{lin.}} mL \text{.}$$  
Therefore 
$$D^{(n)} \equiv_{\mathrm{lin.}} \delta_f^n m^nL$$ 
and
$$
D^{(n)} - ( D^{\operatorname{red}}_{\operatorname{p.i.}})^{(n)} \equiv_{\mathrm{lin.}} \left( \delta^n_f m^n - \sum_{i=1}^{q_n} m_i^{(n)} \right) L \text{.}
$$
Thus
\begin{equation}\label{eqn:2}
\sum_{v \in S} \lambda_v(D^{(n)} - (D^{\operatorname{red}}_{\operatorname{p.i.}})^{(n)},x') \leq \left( \delta_f^n m^n - \sum_{i=1}^{q_n} m_i^{(n)} \right) h_L(x') + \mathrm{O}(1) \text{.}
\end{equation}

Now since 
$$
\lambda_v(D,f^{(n)}(x')) = \lambda_v(D^{(n)},x') + \mathrm{O}(1)
$$
for $v \in M_{\KK}$ and since 
$$
D^{(n)} = D^{(n)} - (D^{\operatorname{red}}_{\operatorname{p.i.}})^{(n)} + (D^{\operatorname{red}}_{\operatorname{p.i.}})^{(n)}
$$
combining Equations \eqref{eqn:1} and \eqref{eqn:2} yields
\begin{equation}\label{eqn:3}
\sum_{v \in S} \lambda_v(D,f^{(n)}(x')) \leq \left(\delta_f^n m^n - \sum_{i=1}^{q_n} m_i^{(n)} + \gamma + \epsilon \right) h_L(x') + \mathrm{O}(1) \text{.}
\end{equation}

On the other hand, since
$$
h_L(x') = \frac{1}{\delta_f^n m^n} h_{\Osh_X(D)}(f^n(x')) + \mathrm{O}(1) 
$$
Equation \eqref{eqn:3} can be rewritten as
\begin{multline}\label{eqn:4}
\sum_{v \in S} \lambda_v(D,f^{(n)}(x')) \leq \left( \frac{\delta_f^n m^n - \sum_{i=1}^{q_n} m_i^{(n)}+\gamma + \epsilon }{\delta_f^n m^n} \right) h_{\Osh_X(D)}(f^{(n)}(x')) \\ + \mathrm{O}(1) \text{ for $x' \in X \setminus Z$.}
\end{multline}

Recall that we have set 
$$
c_n = \frac{ \sum_{i=1}^{q_n} m_i^{(n)} - \gamma }{\delta_f^n m^n} \text{.}
$$
Therefore Equation \eqref{eqn:4} can be re-expressed as
\begin{multline}\label{eqn:5}
\sum_{v \in S} \lambda_v(D,f^{(n)}(x')) \leq \left(1 - c_n + \frac{\epsilon}{\delta^n_f m^n} \right) h_{\Osh_X(D)}(f^{(n)}(x')) \\ + \mathrm{O}(1) \text{ for all $x' \in X \setminus Z$.}
\end{multline}

Now, for all $m \geq n$, let 
$x'= f^{(m-n)}(x)$.  Then, using Equation \eqref{eqn:5}, the conclusion is that if 
$f^{(m-n)}(x) \not \in Z$, 
then 
\begin{multline}\label{eqn:6}
\sum_{v \not \in S} \lambda_v(D,f^{(m)}(x)) \\ > \left(c_n - \frac{\epsilon}{\delta^n_f m^n} \right) h_{\Osh_X(D)}(f^{(m)}(x)) - \mathrm{O}(1) \\ > \left( c_n - \frac{\epsilon}{\delta^n_f m^n} \right) m h_L(f^{(m)}(x)) - \mathrm{O}(1).
\end{multline}

Here, in \eqref{eqn:6}, we have used the relation that 
$$
\sum_{v \in S} \lambda_v(D,f^{(m)}(x)) + \sum_{v \not \in S} \lambda_v(D,f^{(m)}(x)) = h_{\Osh_X(D)}(f^{(m)}(x)) \text{.}
$$

We finally now want to address the question of Zariski non-density of the set
\begin{equation}\label{eqn:7}
\left\{ f^{(m)}(x) : \frac{\sum_{v \not \in S} \lambda_v(D,f^{(m)}(x)) }{m h_L(f^{(m)}(x)) } \leq c_n - \epsilon \right\} \text{.}
\end{equation}

The situation is clear if $\# \mathrm{O}_f(x) < \infty$.  If $\# \mathrm{O}_f(x) = \infty$, then Northcott's Theorem for big line bundles, compare with \cite[Theorem 2.4.9]{Bombieri:Gubler}, implies, upon enlarging $Z$ if necessary, that $h_L(f^{(m)}(x)) \to \infty$ as $m \to \infty$.

So, there is the inequality 
$$
\frac{\sum_{v \not \in S} \lambda_v(D,f^{(m)}(x)) }{m h_L(f^{(m)}(x)) } > \left( c_n - \epsilon' \right) - \mathrm{O}(1)
$$
for $m \gg 0$ and sufficiently small $\epsilon' > 0$ with $f^{(m-n)}(x) \not \in Z$.

Finally, observe that if $f^{(m-n)}(x) \in Z$, then $f^{(m)}(x) \in f^{(n)}(Z)$; further $f^{(n)}(Z) \subsetneq X$ is a proper Zariski closed subset.

Therefore for sufficiently small $\epsilon > 0$, the set \eqref{eqn:7} is contained in the union of $f^{(n)}(Z)$ together possibly with a finite collection of $f$-orbit points.
\end{proof}

\providecommand{\bysame}{\leavevmode\hbox to3em{\hrulefill}\thinspace}
\providecommand{\MR}{\relax\ifhmode\unskip\space\fi MR }
\providecommand{\MRhref}[2]{%
  \href{http://www.ams.org/mathscinet-getitem?mr=#1}{#2}
}
\providecommand{\href}[2]{#2}


\begin{thebibliography}{10}

\bibitem{Bombieri:Gubler}
E.~Bombieri and W.~Gubler, \emph{Heights in {D}iophantine geometry}, Cambridge
  University Press, Cambridge, 2006.

\bibitem{Evertse:Ferretti:2013}
J.-H. Evertse and R.~G. Ferretti, \emph{A further improvement of the
  {Q}uantitative {S}ubspace {T}heorem}, Ann. of Math. (2) \textbf{177} (2013),
  no.~2, 513--590.

\bibitem{Grieve:points:bounded:degree}
N.~Grieve, \emph{On arithmetic inequalities for points of bounded degree}, Res.
  Number Theory \textbf{7} (2021), no.~1, Paper No. 1, 14.

\bibitem{Grieve:MVT:2019}
\bysame, \emph{{O}n {D}uistermaat-{H}eckman measure for filtered linear
  series}, C. R. Math. Acad. Sci. Soc. R. Can. \textbf{44} (2022), no.~1,
  16--32.

\bibitem{Grieve:chow:approx}
\bysame, \emph{Expectations, concave transforms, {C}how weights, and {R}oth's
  theorem for varieties}, Manuscripta Math. \textbf{172} (2023), no.~1--2,
  291--330.

\bibitem{Grieve:qualitative:subspace}
\bysame, \emph{On qualitative aspects of the quantitative subspace theorem},
  Rocky Mountain J. Math. (To appear).

\bibitem{Heier:Levin:2017}
G.~Heier and A.~Levin, \emph{A generalized {S}chmidt subspace theorem for
  closed subschemes}, Amer. J. Math. \textbf{143} (2021), no.~1, 213--226.

\bibitem{Laz}
R.~Lazarsfeld, \emph{Positivity in algebraic geometry {I}}, Springer-Verlag,
  Berlin, 2004.

\bibitem{Meng:Zhang:2018}
S.~Meng and D.-Q. Zhang, \emph{{B}uilding blocks of polarized endomorphisms of
  normal projective varieties}, Adv. Math. (2018), no.~325, 243--273.

\bibitem{Ru:Vojta:2016}
M.~Ru and P.~Vojta, \emph{A birational {N}evanlinna constant and its
  consequences}, Amer. J. Math. \textbf{142} (2020), no.~3, 957--991.

\bibitem{Silverman:1993}
J.~H. Silverman, \emph{{I}nteger points, {D}iophantine approximation, and
  iteration of rational maps}, Duke Math. J. \textbf{71} (1993), no.~3,
  793--829.

\bibitem{Silverman:2013}
\bysame, \emph{Primitive divisors, dynamical {Z}sigmondy sets, and {V}ojta's
  conjecture}, J. Number Theory \textbf{133} (2013), no.~9, 2948--2963.

\bibitem{Yasufuku:2015}
Y.~Yasufuku, \emph{Integral points and relative sizes of coordinates of orbits
  in {$\Bbb{P}^N$}}, Math. Z. \textbf{279} (2015), no.~3--4, 1121--1141.

\bibitem{Zhang:2006}
S.-W. Zhang, \emph{Distributions in algebraic dynamics}, Surveys in
  Differential Geometry {X} (2006), 381--430.

\end{thebibliography}
\end{document}